\documentclass{article}[12pt]
\usepackage{a4}
\usepackage{amsfonts}
\usepackage[utf8]{inputenc}
\usepackage[T1]{fontenc}
\usepackage{amsmath,amssymb}
\usepackage{mathrsfs}
\usepackage[all]{xy}
\usepackage{gastex}
\usepackage{url}
\usepackage{tikz}
\usetikzlibrary{decorations.pathmorphing}
\usetikzlibrary{decorations.markings}
\usetikzlibrary{chains}
\usepgflibrary{arrows}
\usepackage{pgffor}
\usepackage{verbatim}
\usepackage{titlesec}
\titleformat\section{}{}{0pt}{\Large\scshape\bfseries\filcenter\thesection{} - }
\usepackage{lmodern}
\usepackage[all]{xy}
\usepackage[miktex]{gnuplottex}
\usepackage{tikz-cd}
\usepackage{xlop}
\usepackage{multicol}

\usepackage{enumitem}

\usepackage{gastex}
\usepackage{tikz}
\usetikzlibrary{decorations.pathmorphing}
\usetikzlibrary{decorations.markings}
\usetikzlibrary{chains}
\usepgflibrary{arrows}
\usepackage{pgffor}
\usepackage{verbatim}

\usepackage{pifont}

\usepackage{tabularx} 
\usepackage{multirow} 
\usepackage{multicol} 
\usepackage{arydshln} 
\usepackage{fancybox} 
\usepackage{multicol} 
\usepackage{array} 
\usepackage{fancybox}

\usepackage{graphics} 
\usepackage{graphicx} 
\usepackage{pstricks,pst-node} 
\usepackage{tikz} 

\usepackage{amssymb}
\usepackage{amsmath}
\usepackage{amsthm}

\theoremstyle{definition}
\newtheorem{Def}{Définition}[section]

\theoremstyle{plain}
\newtheorem{Prop}[Def]{Proposition}

\newtheorem{Cor}[Def]{Corollaire}
\newtheorem{Th}[Def]{Théorème}
\newtheorem{Lem}[Def]{Lemme}

\def\mod{\mathrm{mod}}

\def\Puis{\mathrm{Puis}}
\def\Gal{\mathrm{Gal}}

\def\mod{\mathrm{mod}}

\def\o{\mathrm{o}}
\def\t{\mathrm{t}}
\def\res{\mathrm{res}}

\def\Stab{\mathrm{Stab}}
\def\limproj{\mathop{\oalign{\hbox{\rm lim}\cr
\hidewidth$\longleftarrow$\hidewidth\cr}}}
\def\limind{\mathop{\oalign{\hbox{\rm lim}\cr
\hidewidth$\longrightarrow$\hidewidth\cr}}}

\font\twelvebb=msbm10 
\font\fivebb=msbm5
\font\sevenbb=msbm7
\newfam\bbfam  \scriptscriptfont\bbfam=\fivebb
\textfont\bbfam=\twelvebb
\scriptfont\bbfam=\sevenbb
\def\bb{\fam\bbfam\twelvebb}

\def\N{{\bb N}}

\def\limproj{\mathop{\oalign{\hbox{\rm lim}\cr
\hidewidth$\longleftarrow$\hidewidth\cr}}}
\def\limind{\mathop{\oalign{\hbox{\rm lim}\cr
\hidewidth$\longrightarrow$\hidewidth\cr}}}

\begin{document}

\tikzset{
    moyen/.style={thick, postaction={decorate},
        decoration={markings,mark=at position .6 with {\arrow[>=stealth]{>}}}},
    moyenn/.style={thick, postaction={decorate},
        decoration={markings,mark=at position .6 with {\arrow[>=stealth]{<}}}},    
    gros/.style={line width=2pt, postaction={decorate},
        decoration={markings,mark=at position .6 with {\arrow[>=stealth]{>}}}},
     gross/.style={line width=2pt, postaction={decorate},
        decoration={markings,mark=at position .6 with {\arrow[>=stealth]{<}}}},
groscourt/.style={line width=2pt, postaction={decorate},
        decoration={markings,mark=at position .7 with {\arrow[>=stealth]{>}}}}
}

\begin{center}
\LARGE Groupes profinis presqu'amalgamés.
\end{center}
\vskip 3mm
\noindent
\centerline{Ndeye Coumba Sarr}
\vskip 2mm
\noindent
\centerline{Université de Caen}
\vskip 1cm
\noindent
{\small {\bf Résumé.--- 
Un des résultats fondamentaux de la théorie de Bass-Serre est le théorème suivant : \textit{un groupe est amalgamé si et seulement si il agit sur un arbre avec comme domaine fondamental un segment}. Dans cet article nous donnons un analogue pour les groupes profinis de ce résultat en utilisant la théorie des prographes de Deschamps et Suarez introduite dans \cite{DesSua}.}
\vskip 3mm
\noindent
{\bf Abstract.--- A fundamental result of Bass-Serre theory is the following theorem: an abstract group is a free product with amalgamation if and only if it acts on a tree with a segment as fundamental domain. In this article, an analogous result for profinite groups will be given, using the theory of prographs of Deschamps and Suarez introduced in \cite{DesSua}. }
\vskip 8mm
\noindent

\section{Introduction}
Dans la théorie de Bass-Serre (exposée dans \cite{AST_1983__46__1_0}), on s'intéresse à l'étude de la structure des groupes agissant sans inversion sur un arbre. L'un des premiers résultats affirme qu'un groupe est libre si et seulement s'il agit librement sur un arbre. Dans leur article \cite{DesSua}, Deschamps et Suarez ont étudié un analogue profini de ces groupes libres : les groupes profinis presque libres ( i.e. possèdant un sous-groupe discret libre dense). Ils les caractérisent en terme d'action prolibre sur un proabre. Moralement, les notions de : prographes, proarbres et proactions sont des passages à la limite projective des notions de graphes, arbres et action dans la situation discrète. 
\vskip 2mm
Un autre cas particulier de la théorie de Bass-Serre est celui des groupes agissant sur un arbre avec comme domaine fondamental un arbre. Ce sont exactement les groupes amalgamés. Un exemple d'application, où le domaine fondamental est un segment, consiste à faire agir $SL_2(\mathbb{Z})$ sur le demi-plan de Poincaré, ce qui permet de retrouver le fait bien connu que $SL_2(\mathbb{Z})$ est le produit amalgamé de $\mathbb{Z}/6\mathbb{Z}$ par $\mathbb{Z}/4\mathbb{Z}$ sur $\mathbb{Z}/2\mathbb{Z}$.
\vskip 2mm 
\noindent
L'objet de cet article est de donner un analogue du théorème principal de \cite{DesSua} pour une autre catégorie de groupes profinis : les groupes profinis presqu'amalgamés ( i.e. possèdant un sous-groupe discret amalgamé dense). Les objets et méthodes utilisés par Deschamps et Suarez dans leur article se placent dans un cadre tout à fait général. Nous nous servons ainsi de ces outils pour obtenir les deux résultats principaux de cet article :
\begin{Th}Soit $\Omega$ un groupe profini de rang dénombrable qui agit à domaine fondamental arbre fini sur un proarbre $\mathcal{P}$.  Alors $\Omega$ est presqu'amalgamé.
\end{Th}
\begin{Th}
Soit $\Omega$ un groupe profini de rang dénombrable, finiement presqu'amalgamé. Alors $\Omega$ agit à domaine fondamental arbre fini sur un proarbre. 
\end{Th}
\vskip 2mm
\noindent
Ici, l'action d'un groupe profini sur un prographe consiste à faire agir les groupes du système projectif de groupes finis associés au groupe profini sur chaque étage d'un système projectif de graphes (prographes), de sorte que l'action à chaque étage possède un domaine fondamental qui est un arbre fini (action à domaine fondamental arbre fini). Nous renvoyons au $§3$ pour des définitions plus précises.
\vskip 2mm
Dans le premier paragraphe de ce texte, nous fixons quelques notations et rappelons les notions introduites dans \cite{DesSua}. Nous y montrons aussi que la limite inductive d'un prographe associé à un amalgame de groupes finis est un arbre (c'est un prographe sylvestre). Dans le paragraphe suivant, nous introduisons la notion de groupe profini presqu'amalgamé et nous démontrons le résultat principal qui caractérise ces groupes. Dans le dernier paragraphe, nous étudions l'exemple du groupe profini $\mathbb{Z}_{p} \rtimes \mathbb{Z}/2 \mathbb{Z}$, on montre qu'il
est presqu'amalgamé. Ensuite, nous utilisons cet exemple pour obtenir une illustration galoisienne, ce qui nous permet de conclure que le groupe $\Gal(\Puis_p(\mathbb{C})/\mathbb{R}((t)))$ est presqu'amalgamé.
\vskip 2mm
\noindent
Je tiens à remercier Bruno Deschamps et Jérôme Poineau pour leurs précieux conseils et leurs suggestions tout au long de la rédaction de cet article.
\section{Résultats préliminaires}
Dans cette section nous rappelons les bases de la théorie combinatoire des groupes profinis de Deschamps et Suarez. Nous nous référons à \cite{DesSua}. Nous démontrons aussi quelques résultats analogues relatifs aux groupes amalgamés et aux actions à domaine fondamental un arbre fini. Avant cela, fixons quelques conventions et  notations. Nous supposons connues les bases de la théorie de Bass-Serre développée dans \cite{AST_1983__46__1_0} et \cite{raghuram2002groups}. 
\vskip 2mm

Dans ce texte, étant donné un graphe orienté $\Gamma$, on notera $\mathcal{S}(\Gamma)$ (resp. $\mathcal{A} (\Gamma))$ l'ensemble des sommets (resp. des arêtes orientées) de $\Gamma$. Pour toute arête orientée $a  \in  \mathcal{A} (\Gamma)$, on notera $\o(a)$ (resp. $\t(a))$ l'origine (resp. le sommet terminal) de $a$. Quand on fera agir un groupe sur  un graphe, il s'agira d'une action à gauche et cette action sera toujours une action de graphes orientés et sans inversion. Un morphisme de graphes $\theta : \Gamma_1 \to \Gamma_2 $ sera vu comme une unique application de l'ensemble  $\mathcal{ S} (\Gamma_1) \cup \mathcal{A} (\Gamma_1)$ vers l'ensemble $\mathcal{ S} (\Gamma_2) \cup \mathcal{A} (\Gamma_2)$.
\vskip 2mm
Par commodité, on notera indifféremment dans ce texte $\overline{x}$ la classe d'équivalence d'un objet $x$, qu'il s'agisse de la classe d'une arête ou d'un sommet dans un graphe quotient, de la classe
d'un élément dans un groupe quotient ou encore même de la classe d'un élément dans une limite inductive d'ensembles. 
Étant donnés $(G_i)_{i\in I}$ une famille de groupe, $A$ un groupe et, pour tout $ i \in I$, $A \to G_i$ un homomorphisme injectif. La limite inductive de la famille $(A,G_i)$ est notée  $G= \ast_AG_i $ et on l'appelle la somme amalgamée des $G_i$ suivant $A$.
\vskip 2mm

Un prographe est un système projectif de graphes $(\Gamma_i, \theta_{ji})_{i\in I}$ indexé sur un ensemble pré-ordonné d'indices $(I,\leq)$ filtrant à droite, tel que pour tout $i\leq j$ le morphisme de graphes $\theta_{ji}$ soit surjectif. Une section graphiquement cohérente d'un prographe est une famille inductive de sections ensemblistes $(s_{ij})_{i\leq j}$ de $(\theta_{ji})_{i\leq j}$ vérifiant que : 
\vskip 2mm
\noindent
$(C_1)$ pour tout $i\in I$ et tout $a \in \mathcal{A}(\Gamma_i)$ il existe $i_0\geq i$ tel que pour tout $k\geq i_0$
$$\begin{array}{lll}
o(s_{ik}(a))&=&s_{i_0k}(o(s_{ii_0}(a)))\\
t(s_{ik}(a))&=&s_{i_0k}(t(s_{ii_0}(a)))\\
\end{array}$$

\vskip 2mm
\noindent
Une section graphiquement cohérente d'un prographe fournit deux systèmes inductifs d'ensembles, un premier, à partir des sommets et un second pour les arêtes. Ce qui permet de construire un graphe "limite inductive" du prographe de départ. 
\vskip 2mm
\noindent
 Soit $(\Gamma_i, \theta_{ji})_{i\in I}$  un prographe muni d'une section graphiquement cohérente $s=(s_{ij})_{i\leq j}$. Les deux ensembles $\limind_s \mathcal{S}(\Gamma_i)$ et $\limind_s \mathcal{A}(\Gamma_i)$ permettent de définir canoniquement les sommets et les arêtes d'un graphe. Ce graphe est appelé limite inductive du prographe $(\Gamma_i, \theta_{ji})_{i\in I}$ relativement à la section graphiquement cohérente $s=(s_{ij})_{i\leq j}$ et est noté $\limind_s \Gamma_i$.
\vskip 2mm
\noindent
Une section graphiquement cohérente $s$ d'un prographe $(\Gamma_i, \theta_{ji})_{i\in I}$ sera dite arborescente si $\limind_s \Gamma_i$ est un arbre. Un prographe sera dit sylvestre s'il possède une section arborescente. On appellera proarbre la donnée d'un prographe sylvestre muni d'une section arborescente.

\vskip 2mm

Le lemme 6 de \cite{DesSua} montre qu'une section graphiquement cohérente d'un prographe permet de reconstruire l'arbre sur lequel un groupe qui est amalgame de groupes finis agit. Nous l'adaptons ici en ajoutant la donnée de sous-arbres finis.
\begin{Lem} 
Considérons un prographe $(\Gamma_n, \theta_{mn})_{n \in \mathbb{N}}$ indexé par $\mathbb{N}$ (muni de l'ordre naturel) et pour tout $n\in \mathbb{N}$ arbre fini $T_n \in \Gamma_n$. Soient $\Gamma$ un graphe, $T$ sous-arbre fini de $\Gamma$, et pour tout $n \in \mathbb{N}$,  $\theta_n : \Gamma \to \Gamma_n$ épimorphisme de graphes vérifiant que pour tout $n$, $\theta_n(T)=T_n$ et que pour tout $ m \geq n$, le diagramme suivant commute :\\
$$
\begin{tikzcd}
 \Gamma \arrow[r, twoheadrightarrow, "\theta_m" ] \arrow[dr, twoheadrightarrow, "\theta_n" ] & \Gamma_m \arrow[d,twoheadrightarrow,"\theta_{mn}" ]\\ 
 & \Gamma_n
 \end{tikzcd}
 $$
Supposons qu'il existe une suite croissante $(\Delta_k)_{k\in \mathbb{N}}$ de parties de $\Gamma$ telle que $\displaystyle \bigcup_{k\in \mathbb{N}}\Delta_k=\Gamma$, pour tout $k \in \mathbb{N}$, $T \subset \Delta_k$, et qu'il existe une suite strictement croissante d'entiers $(n_k)_k$ telle que, pour tout $k\in \mathbb{N}$, $\Delta_{k+1}$ s'injecte (par $\theta_{n_{k+1}}$) dans $\Gamma_{n_{k+1}}$ et se surjecte (par $\theta_{n_k}$) sur $\Gamma_{n_k}$ :
$$
\begin{tikzcd}
 \Delta_{k+1} \arrow[r, hook, "\theta_{n_{k+1}}" ] \arrow[dr, twoheadrightarrow, "\theta_{n_k}" ] & \Gamma_{n_{k+1}} \arrow[d,twoheadrightarrow,"\theta_{n_{k+1}n_k}" ]\\ 
\Delta_k \arrow[u,hook] \arrow[r,hook, "\theta_{n_k}"] & \Gamma_{n_{k}}
 \end{tikzcd}
 $$

Il existe alors une section graphiquement cohérente $s=(s_{nm})_{n <m}$ telle que:

\begin{enumerate}
\item pour tout $n,m \in \mathbb{N}$, $s_{nm}(T_n)=T_m$ et la restriction $s_{nm}|_{T_{n}}$ est un morphisme de graphe ;
\item le morphisme canonique $\sigma : \limind_s \Gamma_n \to \Gamma$ est un isomorphisme ;
\item la famille $(T_n)_n$ définit arbre fini $\overline{T}$ de $\limind_s \Gamma_n$ et $\sigma(T_n)=T$ pour tout $n \in \mathbb{N}$.

\end{enumerate}
\end{Lem}

\begin{proof}[Démonstration]
Commen\c cons par définir, pour $k\geq 0$,une section $\sigma_{n_k}$ de l'application $\theta_{n_k}|_ {\Delta_{k+1}} : \Delta_{k+1} \to \Gamma_{n_k}$ 
$$
\begin{tikzpicture}[baseline= (a).base]
\node[scale=1.5] (a) at (0,0){
\begin{tikzcd}[column sep=4em,row sep=4em]
 \Delta_{k+1} \arrow[r, hook, "\theta_{n_{k+1}}" ] \arrow[dr, twoheadrightarrow, "\theta_{n_k}" ] & \Gamma_{n_{k+1}} \arrow[d,twoheadrightarrow,"\theta_{n_{k+1}n_k}" ]\\ 
\Delta_k \arrow[u,hook] \arrow[r,hook, "\theta_{n_k}"] & \Gamma_{n_{k}} \arrow[ul,dashrightarrow, shift left, "\sigma_{n_k}"]
 \end{tikzcd}
 };
\end{tikzpicture}
 $$
\vskip 2mm
\noindent
$\bullet$ si $x\in \theta_{n_k}(\Delta_k)$, on pose $\sigma_{n_k}(x) =
  (\theta_{n_k}|_ {\Delta_{k}})^{-1}(x)$,
\vskip 2mm
\noindent
$\bullet$ si $x\notin \theta_{n_k}(\Delta_k)$, alors on prend pour
  $\sigma_{n_k}(x)$ un relevé quelconque de $x$ qui soit dans $\Delta_{k+1}$.
 \vskip 2mm
\noindent
  On définit alors une section $s_{n_kn_{k+1}}$ 
$$
\begin{tikzcd}
 \Gamma_{n_{k+1}} \arrow[r, twoheadrightarrow, "\theta_{n_{k+1}n_k}" ]  & \Gamma_{n_k} \arrow[l, bend left, "s_{n_kn_{k+1}}"]\\ 
\end{tikzcd}
 $$
de l'application $\theta_{n_{k+1}n_k}$ en posant $s_{n_kn_{k+1}} = \theta_{n_{k+1}} \circ \sigma_{n_k}$.
\vskip 2mm
\noindent
Comme $T \subset \Delta_k $, on a :
 $$\sigma_{n_k}(T_{n_k})=T.$$
Et on a bien $$s_{n_kn_{k+1}}(T_{n_k})=T_{n_{k+1}}. $$
En effet, d'une part, soient $\alpha \in T_{n_k}$ et $\beta \in T$ tels que $ \alpha=\theta_{n_k}^{-1}(\beta)$, ce qui donne $s_{n_kn_{k+1}}(\alpha)=\theta_{n_{k+1}}(\beta) \in T_{n_{k+1}}$. D'où l'inclusion dans un sens. \\
D'autre part, soient $\gamma  \in T_{n_{k+1}}$ et $\delta \in T$ tels que $\gamma= \theta_{n_{k+1}}(\delta)$, on a $\delta=\sigma_{n_k} \circ \theta_{n_k}(\delta)$ car la restriction $\theta_{n_k}|_T$ à valeur dans $T_{n_k}$ est isomorphisme de graphes. Par conséquent, $\gamma=  \theta_{n_{k+1}} \circ \sigma_{n_k} \circ \theta_{n_k}(\delta)$. Ce qui prouve l'inclusion dans l'autre sens.\\
Soit $a \in T_{n_k}$, on a $s_{n_kn_{k+1}}(a)=\theta_{n_{k+1}}(b)$ où $b \in T$, encore une fois, car la restriction $\theta_{n_k}|_T$ à valeur dans $T_{n_k}$ est isomorphisme de graphes. Ainsi  $s_{n_kn_{k+1}}|_{T_{n_k}}$ est un bien morphisme de graphes.
\vskip 2mm
\noindent
Pour $h\geq 2$, on pose $s_{n_kn_{k+h}} = s_{n_{k+h-1}n_{k+h}} \circ \cdots \circ s_{n_kn_{k+1}} = \theta_{n_{k+h}} \circ \sigma_{n_k}$.
\vskip 2mm
\noindent
La section $(s_{n_kn_{k+h}})_{k,h}$ du sous-prographe $(\Gamma_ {n_k})$ est alors graphiquement cohérente. En effet, si $a\in \mathcal{ A}(\Gamma_{n_k})$, alors il existe $h\geq k$ tel que $\sigma_{n_k} (a), o(\sigma_{n_k} (a)), t(\sigma_{n_k} (a))\in \Delta_h$. Or $\Delta_h$ s'injecte dans $\Gamma_{n_h}$, donc, en vertu de l'égalité $s_{n_kn_{h+m}} = \theta_{n_{h+m}} \circ \sigma_{n_k}$, on en déduit que, pour $m\geq 0$, on a
$$o(s_{n_kn_{h+m}}(a)) = s_{n_hn_{h+m}} (o(s_{n_kn_h}(a)))$$
\vskip 2mm
\noindent
Définissons maintenant, pour un entier $n$ quelconque, une section $s_{n,n+1}$. Pour cela fixons un indice $n_k$, posons $l=n_k$ et considérons les entiers $i$ tel que $n_k\leq l+i\leq n_{k+1}$. Définissons d'abord les sections $s_{l+i,n_{k+1}}$ et $s_{l+i,l+i+1}$ :
$$
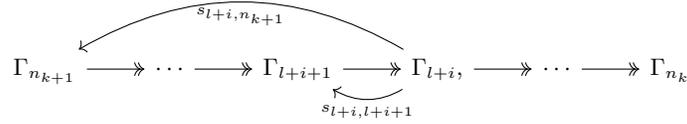
\begin{tikzcd}
 \Gamma_{n_{k+1}} \arrow[r, twoheadrightarrow ]  &\cdots \arrow[r, twoheadrightarrow ]& \Gamma_{l+i+1} \arrow[r, twoheadrightarrow ]&\Gamma_{l+i}\arrow[r, twoheadrightarrow ] \arrow[l,bend left, "s_{l+i,l+i+1}"],\arrow[lll,bend right, "s_{l+i,n_{k+1}}"]& \cdots \arrow[r, twoheadrightarrow ] & \Gamma_{n_k}\\ 
\end{tikzcd}
 $$
par récurrence sur $i$ : pour $i=0$, $s_{l,n_{k+1}}$ est déjà définie et on pose $s_{l,l+1}=\theta_{n_{k+1},l+1}\circ s_{l,n_{k+1}}$. Et
$$s_{l,l+1}(T_l)=\theta_{n_{k+1},l+1}\circ s_{l,n_{k+1}}(T_l)=(\theta_{n_{k+1},l+1}(T_{n_{k+1}}))=T_{l+1}$$
Si pour $i-1\geq 0$, les sections $s_{l+i-1,n_{k+1}}$ et $s_{l+i-1,l+i}$ sont définies, et de sorte que $s_{l+i-1,n_{k+1}}(T_{l+i-1})=T_{n_{k+1}}$ et $s_{l+i-1,l+i}(T_{l+i-1})=T_{l+i}$, on définit $s_{l+i,n_{k+1}}$ de la manière suivante :
\vskip 2mm
\noindent
$\bullet$ si $x\in s_{l+i-1,l+i}(\Gamma_{l+i-1})$, on pose $s_{l+i,n_{k+1}}(x)=s_{l+i-1,n_{k+1}}(\theta_{l+i,l+i-1}(x))$,
\vskip 2mm
\noindent
$\bullet$ si $x\notin s_{l+i-1,l+i}(\Gamma_{l+i-1})$, on prend pour $s_{l+i,n_{k+1}}(x)$ n'importe quel relevé de $x$ dans $\Gamma_{n_{k+1}}$. \\
Ainsi $$s_{l+i,n_{k+1}}(T_{l+i})=T_{n_{k+1}}$$
\vskip 2mm
\noindent
Une fois $s_{l+i,n_{k+1}}$ ainsi définie, on définit comme il se doit $s_{l+i,l+i+1}=\theta_{n_{k+1},l+i+1}\circ s_{l+i,n_{k+1}}$. On a donc $$s_{l+i,l+i+1}(T_{l+i})=\theta_{n_{k+1},l+i+1}\circ s_{l+i,n_{k+1}}(T_{l+i})=\theta_{n_{k+1},l+i+1}(T_{n_{k+1}})=T_{l+i+1}.$$
Il est alors clair que l'on a $s_{l+i,n_{k+1}}=s_{n_{k+1}-1,n_{k+1}}\circ \cdots \circ s_{l+i,l+i+1}$.
\vskip 2mm
\noindent
Si $n\leq m$, on pose $s_{nm}=s_{m-1,m}\circ \cdots \circ s_{n,n+1}$ et le système $(\Gamma_n,s_{nm})_n$ est alors inductif.  La section $s = (s_{nm})_{n\leq m}$ est une section inductive qui possède une sous-section $(s_{n_kn_l})_ {k\leq l}$ graphiquement cohérente. Elle est donc elle-même graphiquement cohérente. On a bien que la restriction  $s_{nm}|_{T_{n}}$ est un morphisme de graphes et
$s_{nm}(T_n)=T_m$.
\vskip 2mm
\noindent
Pour $n_k<l<n_{k+1}$, on définit une application $\sigma_l : \Gamma_l \longrightarrow \Delta_{k+2}$ en posant
$\sigma_l = \sigma_ {n_{k+1}} \circ s_{l, n_{k+1}}$. Les applications $\sigma_l$ commutent visiblement au diagramme inductif et, par propriété universelle, on en déduit l'existence d'une application
$$\sigma : \limind_s \Gamma_n \longrightarrow \Gamma$$
\vskip 2mm
\noindent
Si $x\in \Gamma$, il existe $k\geq 0$ tel que
$x\in \Delta_k$ et alors la classe $\overline{\theta_{n_k}(x)}$ ne dépend pas du choix de $k$. En effet si $k_1<k_2$ sont tels que $x\in \Delta_{k_1}\subset \Delta_{k_2}$, alors par construction on a $s_{n_{k_1}n_{k_2}}\circ \theta_{n_{k_1}}(x)=\theta_{n_{k_2}}(x)$ et donc $\overline{\theta_{n_{k_1}}(x)}=\overline{\theta_{n_{k_2}}(x)}$. Ceci permet de définir une application
$$\psi :\Gamma\longrightarrow \limind_s \Gamma_n$$
en posant $\psi(x)=\overline{\theta_{n_k}(x)}$, o\`u $k$ est tel que $x\in \Delta_k$.\\
Et donc 
$$\psi(I)=\overline{\theta_{n_k}(T)}=\overline{T_{n_k}}$$
\vskip 2mm
\noindent
L'application $\psi$ est un morphisme de graphe. En effet, si $a\in \Gamma$ est une arête et $k$ est un entier tel que $a, o(a),t(a)\in \Delta_k$, alors on a alors pour tout $m\geq n_k$, $o(s_{n_km} (\theta_{n_k}(a))) = s_{n_km}(o
(\theta_{n_k}(a)))$ et donc $o(\psi(a)) = \overline{o(\theta_{n_k}(a))} = \psi(o(a))$. On a de même $t(\psi(a)) = \psi(t(a))$ et $\psi$ est donc bien un morphisme. Il est certainement bijectif puisque $\psi\circ\sigma = {\rm Id}$ et $\sigma\circ\psi = {\rm Id}$.
\vskip 2mm
\noindent
Ainsi, $\sigma$ est un isomorphisme canonique (i.e. provenant de la propriété universelle des objets inductifs) entre $\limind_s \Gamma_n$ et $\Gamma$. Cet isomorphisme dépend toutefois du choix de la suite $(\Delta_k)_k$.
\end{proof}
Par application du précédent lemme au cas d'un groupe agissant sur arbre, nous obtenons le théorème suivant : 
\begin{Th} Soit G un groupe dénombrable agissant sur un arbre $\Gamma$ avec comme domaine fondamental un arbre fini $T$. On suppose que les stabilisateurs des sommets de $T$ sont finis. On considère une suite décroissante $H_0 \supset H_1 \supset \cdots$ de sous-groupes distingués d'indice fini de $G$ telle que $\bigcap \limits _{n \in \mathbb{N}} H_n= \{e\}$ et le prographe $(\Gamma_n, \theta_{mn})_n$ donné  par $ \Gamma_n=H_n \setminus \Gamma$. Pour tout $n \in \mathbb{N}$, on note $\theta_n$ l'épimorphisme canonique de $\Gamma$ dans $\Gamma_n$ et on définit l'arbre $T_n:=\theta_n(T)$. Alors il existe une section graphiquement cohérente, s, de $\theta_{mn}$, telle que $ \limind_s \Gamma_n= \Gamma$ et $s_{nm}(T_n)=T_m $.
\end{Th}
\begin{proof}[Démonstration]
On veut vérifier les hypothèses du lemme précédent et l'appliquer.\\
On note, pour $m \geq n$, $\theta_{mn}: \Gamma_m \to \Gamma_n$ l'épimorphisme canonique induit par l'inclusion $H_m \subset H_n$. On obtient ainsi, pour tout $n \leq m$, le diagramme commutatif suivant:  
$$
\begin{tikzcd}
 \Gamma \arrow[r, twoheadrightarrow, "\theta_m" ] \arrow[dr, twoheadrightarrow, "\theta_n" ] & \Gamma_m \arrow[d,twoheadrightarrow,"\theta_{mn}" ]\\ 
 & \Gamma_n
 \end{tikzcd}
 $$
Nous allons construire la suite $(\Delta_k)_{k\in \mathbb{N}}$ du lemme. Pour cela on considère une partie $\Delta \subset \Gamma$ contenant $T$ et composée d'un nombre fini de sommets et de toutes les arêtes ayant pour origine un de ces sommets. Il existe alors un indice $n \geq 0$ tel que $\Delta$ s'injecte par $\theta_n$ dans $\Gamma_n$. En effet, soient $x_1, \cdots, x_k$ les sommets distincts de $\Delta$. Pour tout $1 \leq i < j \leq n$, il existe $n_{ij} \in \mathbb{N}$ tel que $x_i \neq x_j \mod H_{n_{ij}}$. En effet supposons que pour tout $n \in \mathbb{N}$ $x_i=x_j \mod H_n$.
\vskip 2mm
\noindent
Comme $S(\Gamma)=G.S(T)=\{g_i \cdot P_i | g_i \in G, P_i \in S(T) \}$, on peut écrire tout sommet $x_i \in S(\Gamma)$ sous la forme $ x_i=g_i\cdot P_i $, pour un certain $g_i \in G$ et $P_i \in S(T)$.
\vskip 2mm
\noindent
Ainsi, soient $g_i $ et $g_j$ deux éléments de $G$ et $P_i$ et $P_j$ deux sommets de $\Gamma$ tels que $x_i=g_i\cdot P_i$ et $x_j=g_j\cdot P_j$. Alors $x_i=x_j \mod H_{n}$ implique que $P_i=P_j=:P$, $x_i=g_i\cdot P_i= h_n \cdot x_j=h_ng_j\cdot P_j$, avec $ h_n \in H_n $.\\
On en déduit que $ P= g^{-1}_{i}h_ng_j \cdot P$ et donc que  $ g^{-1}_{i}h_ng_j \in \Stab_G(P)$.
\vskip 2mm
\noindent
Puisque $\Stab_G(P)$ est fini, il existe $a \in \Stab_G(P)$ et une infinité d'indices $n_k$ tels que  $ g^{-1}_{i}h_ng_j=a$, d'où $h_{n_k}=g_iag^{-1}_{j}$. On a $g_iag_j^{-1} \neq e $ car $x_i \neq x_j$ ce qui contredit $\bigcap \limits _{n \in \mathbb{N}} H_n= \{e\}$.
\vskip 2mm
\noindent
Notons $n=\displaystyle{\max_{i,j}n_{ij}}$. Alors pour tous $i<j$, $x_i \neq x_j \mod H_n$.\\
Les sommets de $\Delta$ s'envoient injectivement par $\theta_n$ dans $\Gamma_n$. En utilisant le même argument que pour les sommets,  arêtes de $\Delta$ s'envoient aussi  injectivement par $\theta_n$ dans $\Gamma_n$ car puisque le domaine fondamental est fini il y a un nombre fini d'arêtes qui partent de chaque sommet.\\
Comme $G$ est dénombrable et $T$ aussi alors $\Gamma$ l'est aussi car $\Gamma=G\cdot T$, on peut énumérer $S(\Gamma)$ en une suite $(g_n)_n$. Pour $n\geq 0$, on considère la partie $\Delta_n\subset \Gamma$ constituée des sommets $g_0,\cdots ,g_n$ et de toutes les arêtes de $\Gamma$ qui ont pour origine un des $g_i$ avec $i\leq n$. La suite $(\Delta_n)_n$ est donc une suite de parties de $\Gamma$, croissante pour l'inclusion et de réunion valant $\Gamma$ tout entier. \\
Posons $i_0=0$. D'après ce qui précède, il existe un indice $n_0$ tel que $\Delta_{i_0}$ s'injecte par $\theta_{n_0}$ dans $\Gamma_{n_0}$. Comme $\Gamma_{n_0}$ ne compte qu'un nombre fini de sommets, si l'on relève ces sommets dans $\Gamma^{'}$ il existe alors un indice $i_1>i_0$ tel que $\Delta_{i_1}$ contienne tous ces relevés. La partie $\Delta_{i_1}$ se surjectent alors par $\theta_{n_0}$ sur le graphe $\Gamma_{n_0}$. Il existe alors un indice $n_1>n_0$ tel que $\Delta_{i_1}$ s'injecte par $\theta_{n_1}$ dans $\Gamma_{n_1}$. On construit ainsi, par récurrence, deux suites d'indice $i_0<i_1<\cdots $ et $n_0<n_1<\cdots$ telles que, pour tout $k\geq 0$, $\Delta_{i_{k+1}}$ s'injecte (par $\theta_{n_{k+1}}$) dans $\Gamma_{n_{k+1}}$ et se surjecte (par $\theta_{n_k}$) dans $\Gamma_{n_k}$. \\
On peut alors appliquer le lemme précédent pour conclure : il existe une section graphiquement cohérente $s$ telle que $s_{nm}(T_n)=T_m$ et un isomorphisme canonique de graphes $\sigma:\limind_s \Gamma_n\longrightarrow \Gamma$
\end{proof}
\begin{Cor}
Le prographe d'un amalgame de groupes finis relativement  à une base de filtre de sous-groupes distingués d'indices finis d'intersection réduite au neutre, est un prographe sylvestre. 
\end{Cor}
\begin{proof}[Démonstration]
D'après le théorème $7$ de \cite{AST_1983__46__1_0}, notre groupe amalgamé agit sur un arbre avec comme domaine fondamental un arbre fini. On lui applique alors le théorème précédent pour obtenir le résultat voulu.
\end{proof}
\vskip 2mm
\noindent
Nous allons maintenant aborder la notion de système inductif extrait d'un système projectif et de section pour un système projectif de groupes (cf section 2 de \cite{DesSua}).
\vskip 2mm
On considère un système projectif $(G_i,\varphi_{ji})_{i\in I}$ d'ensembles indexé par un ensemble d'indices $I$ filtrant à droite. On considère pour tout $i\in I$, un sous-ensemble $E_i\subset G_i$ et pour tout $i\leq j$ une section ensembliste $r_{ij}$ de $\varphi_{ji}$. Si le système $(E_i,r_{ij})_i$ est inductif, on dira que c'est un système inductif extrait du système projectif $(G_i,\varphi_{ji})_{i\in I}$. Deschamps et  Suarez montrent que dans ce cas, l'ensemble $\limind_r E_i$ s'injecte canoniquement dans l'ensemble $\limproj_\varphi G_i$.
\vskip 2mm
On s'intéresse maintenant à un système projectif de groupes 
 $(G_i,\varphi_{ji})_{i\in I}$ et on appellera section de ce système tout système inductif extrait $(G_i,r_{ij})_i$ tel que l'image de $\limind_r G_i$ dans $\limproj_\varphi G_i$ soit un sous-groupe.
\begin{Prop}[ \cite{DesSua}, proposition 13 ]
 Soient $(G_i,\varphi_{ji})_{i\in I}$ un système projectif de groupes finis et $G=\limproj_\varphi G_i$. Si $r$ désigne une section de $(G_i,\varphi_{ji})_{i\in I}$ alors l'image $H$ de $\limind_r G_i$ dans $G$ est un sous-groupe dense de $G$. En particulier, $G$ s'identifie à la complétion profinie de $H$ relativement à la base de filtre de sous-groupes distingués d'indices finis $\{H\cap ker(G\rightarrow G_i)\}_{i\in I}$.
\vskip 2mm
\noindent
 Réciproquement, soient $H$ un groupe dénombrable et $H_0\supset H_1\supset \cdots$ une suite décroissante de sous-groupes distingués d'indices finis et d'intersection réduite au neutre. Posons, pour tout entier $n$, $G_n=H/H_n$ et considérons $G=\limproj G_n$. Il existe alors une section $r$ du système projectif $(G_n,\varphi_{mn})_n$ telle que $H\simeq \limind_r G_n$. Plus précisément, si l'on considère les injections canoniques
 $$
\begin{tikzpicture}[baseline= (a).base]
\node[scale=1] (a) at (0,0){
\begin{tikzcd}[column sep=4em,row sep=4em]
H \arrow[r,hook, "\pi"] \arrow[rr,dashrightarrow, bend right, "\omega"]&G &\limind_r G_n \arrow[l, hook, "\theta"]
\end{tikzcd}
 };
\end{tikzpicture}
 $$ 
alors il existe une bijection $\omega: H\longrightarrow \limind_r G_n$ telle que $\pi=\theta\circ \omega$ (en d'autres termes, on peut reconstruire par une certaine section, une copie de $H$ dans sa complétion profinie)
\end{Prop}
Nous choisissons d'inclure la démonstration de Deschamps et Suarez de cette proposition car nous allons l'utiliser dans la suite.
\begin{proof}[Démonstration]
On garde les notations de la preuve de la proposition 10.  Il est clair que les applications $\theta_k$ sont surjective et comme $\theta_k=\varphi_k\circ \theta$ on en déduit que $\varphi_k(\theta(H))=G_k$ pour tout $k$, ce qui montre que $\theta(H)$ est un sous-groupe dense de $G$. Si maintenant, pour $i\in I$ fixé, on note $H_i=H\cap ker(\varphi_i)$, on a alors $G_i=H/H_i$ et donc $G=\limproj H/H_i$ est bien la complétion profinie du groupe $H$ relativement à la base de filtre $\{H_i\}_i$.
\vskip 2mm
\noindent
Réciproquement, si l'on suppose que $H$ est dénombrable, on peut donc énumérer $H$ en une suite $h_1,h_2,\cdots$. Pour tout entier $k\geq 1$, on pose $\Delta_k=\{h_1,\cdots ,h_k\}$. On prend les notations suivantes
$$
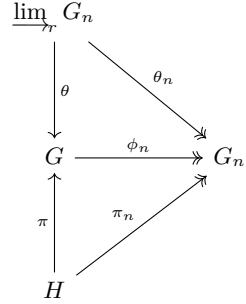
\begin{tikzpicture}[baseline= (a).base]
\node[scale=1] (a) at (0,0){
\begin{tikzcd}[column sep=4em,row sep=4em]
\limind_r G_n \arrow[d, "\theta"] \arrow[rd, twoheadrightarrow, "\theta_n"]\\
G \arrow[r,twoheadrightarrow, "\phi_n"]& G_n\\
H \arrow[u,"\pi"] \arrow[ur,twoheadrightarrow, "\pi_n"]
\end{tikzcd}
 };
\end{tikzpicture}
 $$ 
Exactement comme dans la preuve du lemme 2.5 et du théorème 2.6 (cas des sommets), on construit à partir de la suite de parties $(\Delta_k)_k$ un système inductif extrait $r=(r_{nm})_{n\leq m}$, en particulier, on dispose de deux suites $(i_k)_k$ et $(n_k)_k$ strictement croissantes telles que, par le diagramme suivant
$$
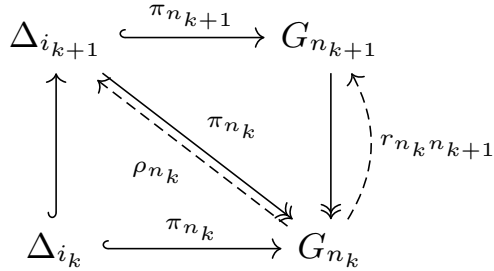
\begin{tikzpicture}[baseline= (a).base]
\node[scale=1.5] (a) at (0,0){
\begin{tikzcd}[column sep=4em,row sep=4em]
\Delta_{i_{k+1}} \arrow[r, hook, "\pi_{n_{k+1}}" ] \arrow[dr, twoheadrightarrow, "\pi_{n_k}" ] & G_{n_{k+1}} \arrow[d,twoheadrightarrow,]\\ 
\Delta_{i_k} \arrow[u,hook] \arrow[r,hook, "\pi_{n_k}"] & G_{n_{k}}  \arrow[ul,dashrightarrow, shift left, "\rho_{n_k}" ]\arrow[u, bend right,dashrightarrow, swap, "r_{n_k n_{k+1}}"]
 \end{tikzcd}
};
\end{tikzpicture} 
$$
la partie $\Delta_{i_k}$ se relève identiquement par $r_{n_kn_{k+1}}$ dans $\Delta_{i_{k+1}}$. Définissons l'application $\omega :H\longrightarrow \limind_r G_n$ de la manière suivante : si $h\in H$ il existe un indice $i_k$ tel que $h\in \Delta_{i_k}$ et on pose
$$\omega(h)=\overline{\pi_{n_k}(h)}$$
La manière dont on a construit $r$ assure que la définition de $\omega(h)$ ne dépend pas du choix de l'indice $i_k$ et donc que $\omega$ est bien définie. L'application $\omega$ est surjective : soit $x=\overline{x_l}\in \limind_r G_n$, et un indice $k$ tel que $n_k\geq l$. L'ensemble $\Delta_{i_{k+1}}$ se surjecte sur $G_l$, on a alors $x=\overline{r_{ln_{k+1}}(x_l)}$ et comme $G_l$ se relève injectivement par $r_{ln_{k+1}}$ dans $\pi_{n_{k+1}}(\Delta_{i_{k+1}})$, il existe (un unique) $h\in \Delta_{i_{k+1}}$ tel que $\pi_{n_{k+1}}(h)=r_{ln_{k+1}}(x_l)$. On a donc $\omega(h)=x$. 
\vskip 2mm
\noindent
Montrons maintenant que $\theta\circ \omega=\pi$. Remarquons pour commencer que si $x\in \limind_r G_n$ et que $x=\overline{x_n}$ pour un certain entier $n$, alors on a 
$$\theta(x)=(\varphi_{n,1}(x_n),\cdots ,\varphi_{n,n-1}(x_n),x_n,r_{n,n+1}(x_n),\cdots)\in \limproj G_n\subset \prod G_n$$
Prenons maintenant $h\in H$ et considérons un indice $i_k$ tel que $h\in \Delta_{i_k}$, on a alors
$$\begin{array}{lll}
\theta\circ \omega(h)&=&\theta (\overline{\pi_{n_k}(h)})\\
&=&(\varphi_{n_k,1}(\pi_{n_k}(h)),\cdots ,\varphi_{n_k,n_k-1}(\pi_{n_k}(h)),\pi_{n_k}(h),r_{n_k,n_k+1}(\pi_{n_k}(h)),\cdots)\\
&=&(\pi_1(h),\cdots ,\pi_{n_k-1}(h),\pi_{n_k}(h),r_{n_k,n_k+1}(\pi_{n_k}(h)),\cdots)\\
\end{array}$$
Maintenant, par construction, pour tout $x\in \Delta_{i_k}$ et tout $l\geq n_k$ on a $r_{n_k,l}(\pi_{n_k}(x))=\pi_l(x)$ et donc
$$\theta\circ \omega(h)=(\pi_1(h),\cdots ,\pi_{n_k-1}(h),\pi_{n_k}(h),\pi_{n_k+1}(h),\cdots)=\pi(h)$$
L'application $\pi$ étant injective, $\omega$ l'est aussi ce qui prouve finalement qu'elle est bien bijective.
\end{proof}

\section{Groupes profinis presqu'amalgamés}
Nous commençons cette section par une définition d'action profinie donnée par Deschamps et Suarez (\cite{DesSua} définition $15$ du $2.2.$). On verra ensuite qu'une telle action induit une action d'un sous groupe dense du groupe profini sur un graphe limite inductive relativement à une section graphiquement cohérente et une section pour le système projectif de groupe.
Nous finirons en abordant la notion de groupe profini presqu'amalgamé et démontrons les deux théorèmes principaux.
\begin{Def}[Deschamps et Suarez]
Soit ${\cal P}=(\Gamma_i, \theta_{ji})_{i\in I}$ un prographe et $G$ un groupe profini. On dit que $G$ agit sur ${\cal P}$ si :
\vskip 2mm
\noindent
$\bullet$ Il existe une base de filtre, $(U_i)_{i\in I}$, de sous-groupes ouverts distingués de $G$, indexée par $(I,\leq)$ ($U_j\subset U_i\Longleftrightarrow i\leq j$) tel que $\bigcap_i U_i=\{e\}$.
\vskip 2mm
\noindent
$\bullet$ Pour tout $i\in I$, le groupe quotient $G_i=G/U_i$ agit sur le graphe $\Gamma_i$.
\vskip 2mm
\noindent
$\bullet$ Les actions des groupes $G_i$ sur les graphes $\Gamma_i$ sont compatibles pour l'ordre : pour tout $i\leq j$, $g\in G_j$ et $x\in \Gamma_j$, on a $\theta_{ji}(g \cdot x)=\varphi_{ji}(g) \cdot \theta_{ji}(x)$ o\`u $\varphi_{ji}:G_j\longrightarrow G_i$ est la surjection canonique induite par l'inclusion $U_j\subset U_i$.
\vskip 2mm
\noindent
 Si l'on dispose d'une section graphiquement cohérente $s$ de $(\Gamma_i, \theta_{ji})_{i\in I}$ et d'une section $r$ du système projectif $(G/U_i,\varphi_{ji})_{i\in I}$ on dira que l'action de $G$ sur $\cal P$ est complètement compatible à $s$ et $r$ si l'on a, en outre, la propriété suivante : 
$$(C_3)\ \forall i\in I,\ \forall (g_i,a_i)\in G_i\times \Gamma_i,\ \exists i_0\geq i,\ \forall k\geq i_0,\ s_{i_0k}\left( r_{ii_0}(g_i) \cdot {s_{ii_0}(a_i)} \right)=r_{ik}(g_i) \cdot s_{ik}(a_i)$$
\end{Def}
\vskip 2mm
\noindent
 On reprend les notations de la définition précédente et l'on suppose que l'action de $G$ sur $\cal P$ est complètement compatible à $s$ et $r$. Considérons un indice $i\in I$ et $(a_i,g_i)\in \Gamma_i\times G_i$. Soit $i_0$ un indice satisfaisant à la condition $(C_3)$. On pose :
$$g \cdot a=\overline{r_{ii_0}(g_i) \cdot s_{ii_0}(a_i)}$$
Ceci ne dépend que de $a$ et $g$.
\vskip 2mm
\noindent
 Dans la proposition $16$ de \cite{DesSua}, Deschamps et Suarez montrent que l'application $(a,g)\longmapsto  g \cdot a$ définie précédemment est une action du groupe $\limind_r G_i$ sur le graphe $\limind_s \Gamma_i$.

\begin{Def} On dira qu'un groupe profini agit profiniement sur un prographe $\mathcal{P}$ si $G$ agit de manière complètement compatible à deux sections $s$ et $r$. Si de plus, pour tout $i\in I$, $G_i$ agit sur $\Gamma_i$ avec comme domaine fondamental un arbre fini $T_i$, avec $s_{ij}(T_i)=T_j$ et la restriction  $s_{ij}|_{T_i}$ est un morphisme de graphe, et enfin  la condition :
$$(C_4) \forall i \in I, \forall \alpha_i \in T_i,\forall g_i \in G_i, \exists i_0 \geq i, \forall k \geq i_0, \ s_{ik}(g_i \cdot \alpha_i)=r_{ik}(g_i) \cdot \alpha_k$$
 est vérifiée, on dira que $G$ agit \textbf{à domaine fondamental arbre fini} (dfaf) sur le prographe $\mathcal{P}$.
\end{Def}
\begin{Lem} Soit $ G$ un groupe profini agissant à domaine fondamental arbre fini sur un proarbre $\mathcal{P} $, alors $\limind_r G_i$ agit sur $\limind_s \Gamma_i$ avec comme domaine fondamental un arbre fini $T:=\overline{T_i}$, pour tout $i \in I$.
\end{Lem}
\begin{proof}[Démonstration]Soit $a \in \limind_s \Gamma_i$, $a=\overline{a_i}$ où $a_i \in \Gamma_i$. Comme $G$ agit à domaine fondamental arbre fini sur $\mathcal{P}$ il existe donc $g_i\in G_i$ et il existe $\alpha_i \in T_i$, tels que $a_i=g_i \cdot \alpha_i$. On cherche $g \in \limind_r G_i$  et $\alpha \in T$ tel que $a=g \cdot \alpha$. Les éléments $g:=\overline{g_i}$ et $\alpha = \overline{\alpha_i}$ conviennent. En effet :
$$\begin{array}{lll}
\overline{g_i} \cdot \overline{\alpha_i}&=&\overline{r_{ii_0}(g_i)\cdot s_{ii_0}(\alpha_i)} \\
&=&\overline{r_{ii_0}(g_i).\alpha_{i_0}}\\
&=& \overline{s_{ii_0}(g_i \cdot \alpha_i)}\\
&=& \overline{s_{ii_0}(a_i)} \\
&=& \overline{a_i}\\
&=& a
\end{array}$$
T est bien un domaine fondamental pour l'action de $\limind_r G_i$ sur $\limind_s \Gamma_i$ et c'est aussi un arbre fini.
\end{proof}
\vskip 2mm
\noindent
Dans la  suite on s'intéressera qu'aux groupes amalgamés avec un nombre  fini de facteurs. Nous introduisons ainsi la convention/définition suivante :
\begin{Def}~\\
\begin{enumerate}
\item Un groupe G est dit amalgamé s'il est amalgame d'un nombre fini de facteurs.
\item Un groupe profini sera dit presque-amalgamé s'il est égal à la complétion profinie d'un groupe amalgamé relativement à une base de filtre de sous-groupes distingués d'indice fini $\mathcal{ B}$, vérifiant $\displaystyle \bigcap_{H\in {\mathcal{ B}}}H=\{e\}$.
\end{enumerate}
\end{Def}
L'hypothèse de séparation $\displaystyle \bigcap_{H\in {\mathcal{ B}}}H=\{e\}$ nous assure que le groupe amalgamé s'injecte de manière dense dans sa complétion. 
\vskip 2mm 
\noindent
Comme pour le cas de la presque liberté (cf lemme $2$ de \cite{DesSua}), on a de manière plus précise le lemme suivant :
\vskip 2mm
\noindent
\begin{Lem} Un groupe profini est presqu'amalgamé si et seulement s'il possède un sous-groupe (abstrait) amalgamé dense.
\end{Lem}
\begin{Def} Un groupe profini est dit \textbf{finiement} presqu'amalgamé s'il possède un sous groupe dense, amalgame de groupes finis.
\end{Def}
\vskip 2mm 
\noindent
Nous obtenons un analogue du théorème principal de \cite{DesSua} donné par les deux prochains théorèmes.
\begin{Th}Soit $\Omega$ un groupe profini de rang dénombrable qui agit à domaine fondamental arbre fini sur un proarbre $\mathcal{P}$. Alors $\Omega$ est presqu'amalgamé.
\end{Th}
\begin{proof}[Démonstration] Par hypothèse $\Omega= \limproj_i G_i$ agit à domaine fondamental segment sur un proarbre $\mathcal{P}=(\Gamma_i, \theta_{ji})_i$ relativement à deux sections $s$ et $r$. Le sous-groupe $\limind_r G_i$ dense agit d'après le lemme $3.3$ sur l'arbre $\limind_s \Gamma_i$ avec comme domaine fondamental un arbre fini $T$. Il est donc amalgamé d'après $\cite{AST_1983__46__1_0}$. Ainsi $\Omega $ possède un sous groupe abstrait dense qui est amalgamé, il est donc presqu'amalgamé.
\end{proof}
\noindent
Nous allons maintenant démontrer une réciproque partielle du théorème précédent. Nous aurons besoin du lemme technique suivant.
\begin{Lem} Soit H un groupe dénombrable agissant sur un arbre $\Gamma$ avec comme domaine fondamental un arbre fini $T$. On suppose que les stabilisateurs des sommets de $T$ sont finis. On considère une suite décroissante $H_0 \supset H_1 \supset \cdots$ de sous-groupes distingués d'indice fini de $H$ telle que $\bigcap \limits _{n \in \mathbb{N}} H_n= \{e\}$ et le prographe $(\Gamma_n, \theta_{mn})_n$ donné  par $ \Gamma_n=H_n \setminus \Gamma$. Posons pour tout entier $n$  $G_n= H/H_n$. Pour tout $n \in \mathbb{N}$, on note $\pi_n$ l'épimorphisme canonique de $H$ dans $G_n$ et on suppose que $G_n$ agit $\Gamma_n$ avec comme domaine fondamental l'arbre fini $T_i:=\overline{T}$. Considérons enfin $\Lambda \subset H$ une partie finie de $H$. Il existe alors deux indices :
\begin{itemize}
\item $n \geq 0$ tel que $\Lambda$ s'injecte par $\pi_n$ dans $G_n$ ;
\item $m \geq 0 $ tel que  si $\lambda \in \Lambda$ et $\alpha \in T$ et $ \lambda \notin \Stab(\alpha)$ alors $\overline{\lambda} \notin \Stab(\overline{\alpha})$, où $\overline{\alpha}=\alpha_{m} \in T_m$.
\end{itemize}
\end{Lem}
\begin{proof}[Démonstration]
Nous allons montrer qu'il existe un indice $m \in \mathbb{N}$,  tel que si $\lambda \in \Lambda$, $\alpha \in T $ et $\lambda \notin \Stab(\alpha)$ alors $\overline{\lambda} \notin \Stab(\overline{\alpha})$, où $\overline{\alpha}=\alpha_{m}$.\\
Soit $\lambda \in \Lambda$ avec $ \lambda \notin \Stab(\alpha)$,  on suppose que pour tout $m \in \mathbb{N}$, $\overline{\lambda} \in \Stab(\alpha_m)$. Dans ce cas on a :
$$\begin{array}{lll}
\overline{\lambda} \cdot \alpha_m &=&  \alpha_m\\
\overline{\lambda} \cdot \overline{\alpha} & = &\overline{\alpha}\\
\overline{\lambda \cdot \alpha}&=&\overline{\alpha}
\end{array}$$
Ce qui implique l'existence d'un $h_m \in H_m$ tel que $\lambda\cdot \alpha= h_m \cdot \alpha$, on a donc que $h^{-1}_m \lambda\in \Stab(\alpha)$ ce qui veut dire qu'il existe un élément $a_m \in \Stab(\alpha)$ tel que $h_m^{-1} \lambda=a_m$. Comme $\Stab(\alpha)$ est fini, il existe une infinité d'indices $m_k$ telle  que $a_{m_k}=a \in \Stab(\alpha)$. On a donc $h_{m_k}^{-1}= \lambda^{-1}a \neq e$ d'après l'hypothèse $\lambda \notin \Stab(\alpha)$. Et le fait que $h_{m_k} \neq e$ contredit l'hypothèse $\bigcap \limits _{n \in \mathbb{N}} H_n= \{e\}$.
\vskip 2mm
\noindent
Pour la première affirmation, on raisonne de la même manière.
\end{proof}
\begin{Th}
Soit $\Omega$ un groupe profini de rang dénombrable, finiement presqu'amalgamé. Alors $\Omega$ agit à domaine fondamental arbre fini sur un proarbre. 
\end{Th}
\begin{proof}[Démonstration]
 Comme $\Omega$ est de rang dénombrable, il existe une filtration $U_0\supset U_1\supset \cdots $ de sous-groupes ouverts distingués de $\Omega$. Si l'on considère les groupes quotients $G_n=\Omega/U_n$, alors $\Omega=\limproj_n G_n$. Soit $F$ un sous-groupe amalgamé et dense de $\Omega$. D'après le théorème $7$  de \cite{AST_1983__46__1_0}, $F$ agit sur un arbre $\Gamma$ avec comme domaine fondamental un arbre fini $T$. On considère le prographe $\mathcal{P}=(\Gamma_n)_n$ de $F$ (relativement à la base de filtre $\left\{ F\cap U_n\right\}_n$), donné par $\Gamma_n=F\cap U_n \setminus \Gamma$.
 \vskip 2mm
\noindent
Puisque $F$ agit sur $\Gamma$ avec comme domaine fondamental un arbre fini, on en déduit que pour tout entier $n$, $G_n= F/ F\cap U_n$ agit sur $ \Gamma_n$ avec comme domaine fondamental un arbre fini $T_n$.
\vskip 2mm
\noindent
Le groupe $F$ est dénombrable et $\left\{ F\cap U_n\right\}_n$ est une base de filtre de sous-groupes distingués d'indices fini et d'intersection réduite au neutre. Nous allons construire une section $r$ du système projectif $(G_n, \phi_{mn})_n$, tel que $F \simeq \limind_r G_n$.\\
Soit $\Lambda \subset F$ une partie finie de $F$. Il existe, d'après le lemme précédent deux indices :
\begin{itemize}
\item $n \geq 0$ tel que $\Lambda$ s'injecte par $\pi_n$ dans $G_n$ ;
\item $n' \geq 0 $ tel que si $\lambda \in \Lambda$, $\alpha \in T$ et $ \lambda \notin \Stab(\alpha)$ alors $\overline{\lambda} \notin \Stab(\overline{\alpha})$, où $\overline{\alpha}=\alpha_{n'}$.
\end{itemize}
Puisque $F$ est dénombrable, pour $n \geq 0$, on considère $\Lambda_n=\{f_1,\cdots,f_n\}$, où $f_i \in F$ pour $i=1,\cdots,n$. La suite $(\Lambda_n)_n$ est une suite croissante de parties de $F$ dont la réunion vaut $F$, $\cup_n \Lambda_n=F$.
\vskip 2mm
\noindent
Posons $i_0=0$, $\Lambda_{i_0}=\{ e \}$. Par ce qui précède il existe un indice $n_0$ tel que si $\lambda \in \Lambda$, $ \alpha \in T$ et $ \lambda \notin \Stab(\alpha)$ alors $\overline{\lambda} \notin \Stab(\overline{\alpha})$ et on a que $\Lambda_{i_0}$ s'injecte par $\pi_{n_0}$ dans $G_{n_0}$. Comme $G_{n_0}$ est fini, si l'on relève ses éléments dans $F$, il existe un  indice $i_1 > i_0$  tel que $\Lambda_{i_1}$ contienne tous ces relevés ; la partie $\Lambda_{i_1}$ se surjecte par $\pi_{n_0}$ dans $G_{n_0}$. Ensuite il existe deux indices :
\begin{itemize}
\item $n'_{1} >n_0$ tel que $\Lambda_{i_1}$ s'injecte par $\pi_{n'_1}$ dans $G_{n'_1}$ ;
\item $n''_1 \geq 0 $ tel que si $\lambda \in \Lambda$, $\alpha \in T$  et $ \lambda \notin \Stab(\alpha)$ alors $\overline{\lambda} \notin \Stab(\alpha_{n''_1})$.
\end{itemize}
On prend $n_1=\max(n'_1,n''_1)$. Alors les deux conditions précédentes soient vérifiées pour l'indice $n_1$. On construit ainsi par récurrence deux suites d'indices $i_0 <i_1< \cdots $ et $n_0<n_1< \cdots $ telles que :
\begin{enumerate}[label=\roman*)]
\item $\Lambda_{i_k}$ s'injecte par $\pi_{n_k}$ dans $G_{n_k}$ ;
\item $\Lambda_{i_k}$ se surjecte par $\pi_{n_{k-1}}$ dans $G_{n_{k-1}}$ ;
\item Si $\lambda \in \Lambda_{i_k}$, $\alpha \in T$ et $\lambda \notin \Stab(\alpha)$ alors $\overline{\lambda} \notin \Stab(\alpha_{n_k})$. 
\end{enumerate}
Comme dans la preuve du lemme 2.1 on construit un système inductif $r=(r_{nm})_{n \leq m}$. En particulier on obtient le diagramme suivant :
$$
\begin{tikzpicture}[baseline= (a).base]
\node[scale=1.5] (a) at (0,0){
\begin{tikzcd}[column sep=4em,row sep=4em]
 \Lambda_{k+1} \arrow[r, hook, "\pi_{n_{k+1}}" ] \arrow[dr, twoheadrightarrow, "\pi_{n_k}" ] & G_{n_{k+1}} \arrow[d,twoheadrightarrow,]\\ 
\Lambda_k \arrow[u,hook] \arrow[r,hook, "\pi_{n_k}"] & G_{n_{k}}  \arrow[ul,dashrightarrow, shift left, "\rho_{n_k}" ]\arrow[u, bend right,dashrightarrow, swap, "r_{n_k n_{k+1}}"]
 \end{tikzcd}
};
\end{tikzpicture} 
$$
$r$ est bien une section du système projectif $(G_n, \phi_{mn})_n$, plus précisément, d'après la proposition 2.4 si l'on considère les injections canoniques
 $$
\begin{tikzpicture}[baseline= (a).base]
\node[scale=1] (a) at (0,0){
\begin{tikzcd}[column sep=4em,row sep=4em]
H \arrow[r,hook, "\pi"] \arrow[rr,dashrightarrow, bend right, swap, "\omega"]&G &\limind_r G_n \arrow[l, hook', swap, "\theta"]
\end{tikzcd}
 };
\end{tikzpicture}
 $$ 
alors il existe une bijection $\omega: H\longrightarrow \limind_r G_n$ telle que $\pi=\theta\circ \omega$.\\
Construisons une section graphiquement cohérente $s$ du prographe $\mathcal{P}$. Posons $\Delta_k= \Lambda_k \cdot T$, pour tout $k\in \mathbb{N}$. On a aussi $\Gamma_{n_k}=G_{n_k}\cdot T_{n_k}$. On obtient alors le diagramme suivant :
$$
\begin{tikzpicture}[baseline= (a).base]
\node[scale=1.5] (a) at (0,0){
\begin{tikzcd}[column sep=4em,row sep=4em]
 \Delta_{k+1} \arrow[r, hook, "\theta_{n_{k+1}}" ] \arrow[dr, twoheadrightarrow, "\theta_{n_k}" ] & \Gamma_{n_{k+1}} \arrow[d,twoheadrightarrow,]\\ 
\Delta_k \arrow[u,hook] \arrow[r,hook, "\theta_{n_k}"] & \Gamma_{n_{k}} \arrow[ul,dashrightarrow, shift left, "\sigma_{n_k}"]\arrow[u, bend right,dashrightarrow, swap, "s_{n_k n_{k+1}}"]
 \end{tikzcd}
 };
\end{tikzpicture}
 $$
où $\sigma_{n_k}(g_{n_k} \cdot \alpha_{n_k})=\rho_{n_k}(g_{n_k}) \cdot \alpha$ et $\theta_{n_k}(h_k \cdot \alpha)=\pi_{n_k}(h_k) \cdot \alpha_{n_k}$, pour tout $k \in \mathbb{N}$.
\\
Ces deux applications sont bien définies grâce à la condition $iii)$.
\vskip 2mm
\noindent
Posons :
$$\begin{array}{ccc}
s_{n_kn_{k+1}}(g_{n_k} \cdot \alpha_{n_k})&=&\theta_{n_{k+1}} \circ \sigma_{n_k}(g_{n_k} \cdot \alpha_{n_k}) \\
&=&\theta_{n_{k+1}}(\rho_{n_k}(g_{n_k}) \cdot \alpha) \\
&=& \pi_{n_{k+1}}(\rho_{n_k}(g_{n_k})) \cdot \alpha_{n_{k+1}}\\
&=& r_{n_kn_{k+1}}(g_{n_k})\cdot \alpha_{n_{k+1}} 
\end{array}$$
et $s$ aussi est bien définie. D'après la preuve du lemme 2.1 $s$ est une section graphiquement cohérente du prographe $\mathcal{P}$, on a que $ \limind_s\Gamma_n \simeq \Gamma$ et $s_{nm}(\alpha)=\alpha_m$, avec $\alpha_i \in T_i$, la restriction de $s_{nm}$ à $T_n$ est un morphisme de graphe. L'égalité $s_{n_kn_{k+1}}(g_{n_k} \cdot \alpha_{n_k})= r_{n_kn_{k+1}}(g_{n_k})\cdot \alpha_{n_{k+1}} $ correspond à la condition $(C_4)$.\\
Nous allons donc montrer que la condition ($C_3$) de la définition 15 est vérifiée. Dire que $r$ est une section équivaut à la condition $(C_2)$ suivante : $\forall n \in \mathbb{N}, \forall g_n,g'_n,$ $\exists n_0 \geq n,$ $ \forall m \geq n_0, $
$$ r_{n_0m}(r_{nn_0}(g_n)\cdot r_{nn_0}(g'_n))=r_{nm}(g_n)r_{nm}(g'_n)$$
On agit sur $\alpha_m$ des deux côtés de l'égalité et d'après l'égalité  $ s_{n_kn_{k+1}}(g_{n_k} \cdot \alpha_{n_k}) = r_{n_kn_{k+1}}(g_{n_k})\cdot \alpha_{n_{k+1}}$, on obtient :
$$s_{n_0m}(r_{nn_0}(g_n)r_{nn_0}(g'_n)\cdot \alpha_n)=r_{nm}(g_n) \cdot s_{n_0m}(g'_n\cdot \alpha_n)$$
On a bien la condition $(C_3)$: 
$$s_{n_0m}(r_{nn_0}(g_n) \cdot s_{nn_0}(g'_n \cdot \alpha_n))= r_{nm}(g_n) \cdot s_{n_0m}(g'_n\cdot \alpha_n)$$
\end{proof}
\section{Une illustration galoisiennne}
On considère le groupe diédral infini $\mathcal{D}_{\infty} = \mathbb{Z} \rtimes \mathbb{Z}/2 \mathbb{Z}$ et les éléments $c=(0,1)$ et $\sigma= (1,0)$ qui engendrent le groupe.  Par interprétation géométrique, on peut dire que $\sigma=(1,0)$ est une rotation et $c= (0,1)$ une symétrie axiale. On voit que les éléments $c$ et $\sigma c=(1,1)$ engendrent librement $\mathcal{D}_{\infty}$. Si bien que $\mathcal{D}_{\infty}$ peut-être vu comme le produit libre $\mathbb{Z}/2\mathbb{Z}* \mathbb{Z}/2\mathbb{Z}$. Ainsi $\mathcal{D}_{\infty}$ est un amalgame de deux copies de $\mathbb{Z}/2\mathbb{Z}$ sur le sous-groupe trivial.
\vskip 2mm
\noindent
On se donne $p$ un nombre premier et on considère la base de filtre $((p^n\mathbb{Z},0))_{n\geq 0}$ de sous-groupes distingués dans $\mathcal{D}_{\infty}$ vu comme un produit semi-direct et on note $U_n=(p^n\mathbb{Z},0)$. Pour tout $n \in \mathbb{N}^*$, le quotient $\mathcal{D}_{\infty}/U_n$ est égal au groupe diédral $D_{p^n}$. La famille des groupes $(D_{p^n})_n$ forment un système projectif dont la limite projective est le groupe $\widehat{D_{p^{\infty}}}= \mathbb{Z}_{p} \rtimes \mathbb{Z}/2 \mathbb{Z}$. Par conséquent $ \widehat{D_{p^{\infty}}}$ est la complétion profinie de $\mathcal{D}_{\infty}$ relativement à la base de filtre $(U_n)_n$. Puisque $D_{\infty}$ est dense dans $\widehat{D_{p^{\infty}}} $, ce dernier est donc un groupe profini presqu'amalgamé.
\vskip 2mm
\noindent
Dans cette partie nous allons explicitement retrouver ce résultat en appliquant le résultat  (théorème $3.7$) de la section précédente.
\vskip 2mm

Pour tout entier naturel $ n$ non nul, considérons le groupe diédral $D_{p^n}= \mathbb{Z}/p^n\mathbb{Z} \rtimes \mathbb{Z}/2\mathbb{Z}$ et le système projectif $(D_{p^n}, \varphi_{mn})_{n \in \mathbb{N}}$, dont la limite projective est le groupe profini $\widehat{D_{p^{\infty}}}= \mathbb{Z}_p \rtimes \mathbb{Z}/2\mathbb{Z}$. Pour $m \geq n $, le morphisme $ \varphi_{mn} :  D_{p^m} \to D_{p^n}$  est défini par :
$$ \varphi_{mn}(x \ \mod \ (p^m),a)=(x \ \mod \ (p^n),a).$$
Nous allons maintenant construire un proarbre sur lequel notre groupe profini $\widehat{D_{p^{\infty}}}$ agira.
\vskip 2mm
\noindent
Nous partons de  $ \mathcal{D}_{\infty} = \mathbb{Z} \rtimes \mathbb{Z}/2 \mathbb{Z}$ et on considère les éléments $g_1= (0,1)$ et $g_2= (1,1)$. On pose $G_1=<g_1> \simeq \mathbb{Z}/2 \mathbb{Z} $ et $G_2 =<g_2>  \simeq \mathbb{Z}/2 \mathbb{Z} $.
On obtient les classes suivant les sous-groupes $G_1$ et $G_2$ :
$\mathcal{D}_{\infty}/G_1=\{\{(n,0), (n,1)\} / n \in \mathbb{Z} \}$ et $\mathcal{D}_{\infty}/G_2=\{\{(n,0), (n+1,1)\} / n \in \mathbb{Z} \}.$
\vskip 2mm
\noindent
Soit $\Gamma$ l'arbre dont l'ensemble des sommets $
\mathcal{S}(\Gamma) =\mathcal{D}_{\infty}/G_1 \sqcup \mathcal{D}_{\infty}/G_2\ $ et l'ensemble des arêtes $\mathcal{A}(\Gamma)= D_\infty$. Pour $n \in \mathbb{Z}$, on pose $P_n =\{(n,0), (n,1)\}$ et $Q_n = \{(n,0), (n+1,1)\}$, ce sont les sommets de $\Gamma$. 
\vskip 2mm
\noindent
Soit maintenant $(n,g)\in  A(\Gamma)= D_\infty$. On a $(\o(n,g))=(n,g)\ \mod \ (G_1)= \{(n,0), (n,1)\}=P_n$ et 
$$
\t(n,g) = \left\{
    \begin{array}{ll}
         \{(n,0), (n+1,1)\}=Q_n& \mbox{si } g=0 \\
        \{(n-1,0), (n,1)\}=Q_{n-1} &\mbox{si } g=1 .
    \end{array}
\right.
$$
On en déduit que l'arbre $\Gamma$ est : 
\vskip 2mm
\noindent
\begin{center}
\begin{tikzpicture}
  \foreach \x in {0}
    \draw [moyen] (1.5*\x+0.1,0) -- (1.5*\x+1.4,0)
      (1.5*\x,0) node {\hbox{\tiny $\bullet$}} node[sloped,above=1pt] {$P_{\x}$};
   \foreach \x in {-2}
    \draw [moyen] (1.5*\x+0.1,0) -- (1.5*\x+1.4,0)
      (1.5*\x,0) node {\hbox{\tiny $\bullet$}} node[sloped,above=1pt] {$P_{-1}$};  
  \foreach \x in {2}
    \draw [moyen] (1.5*\x+0.1,0) -- (1.5*\x+1.4,0)
      (1.5*\x,0) node {\hbox{\tiny $\bullet$}} node[sloped,above=1pt] {$P_{1}$};       
    \draw [moyen]  (-3.1,0)--(-4.4,0)  (-4.7,0) node {$\cdots$} (6.2,0) node {$\cdots$};
 \foreach \x in {-1}
     \draw [moyen]  (1.5*\x+1.4,0) -- (1.5*\x+0.1,0) 
      (1.5*\x,0) node {\hbox{\tiny $\bullet$}} node[sloped,above=1pt] {$Q_{\x}$};
 \foreach \x in {1}
     \draw [moyen]  (1.5*\x+1.4,0) -- (1.5*\x+0.1,0) 
      (1.5*\x,0) node {\hbox{\tiny $\bullet$}} node[sloped,above=1pt] {$Q_{0}$};  
 \foreach \x in {3}
     \draw [moyen]  (1.5*\x+1.4,0) -- (1.5*\x+0.1,0) 
      (1.5*\x,0) node {\hbox{\tiny $\bullet$}} node[sloped,above=1pt] {$Q_{1}$};        
   
\end{tikzpicture}
\end{center}
\vskip 2mm
\noindent 
Pour $n \geq 1$ variable,  les graphes quotient $\Gamma_n=\Gamma/U_n$ ci-après  
\vskip 2mm
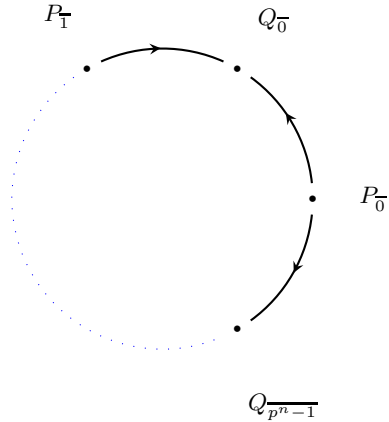
\noindent
\begin{center}
\begin{tikzpicture}[scale=2]
  \foreach \x in {-1}
    \path[draw,moyenn] (0,0) +(60*\x+6:1cm) arc (60*\x+6:60*\x+54:1cm)
      (60*\x:1cm) node {\hbox{\tiny $\bullet$}} ++(60*\x:6mm) node {$ Q_{ \overline {p^n \x}}$};
  \foreach \x in {0}
    \path[draw,moyen] (0,0) +(60*\x+6:1cm) arc (60*\x+6:60*\x+54:1cm)
      (60*\x:1cm) node {\hbox{\tiny $\bullet$}} ++(54*\x:4mm) node {$ P_{\overline \x}$};
    \foreach \x in {1}
    \path[draw,moyenn] (0,0) +(60*\x+6:1cm) arc (60*\x+6:60*\x+54:1cm)
      (60*\x:1cm) node {\hbox{\tiny $\bullet$}} ++(54*\x:4mm) node {$ Q_{\overline 0}$};
    \path[draw] (0,0) ++(60*2:1cm) node {\hbox{\tiny $\bullet$}} +(60*2:4mm) node {$ P_{ \overline 1}$};
    \draw[blue,loosely dotted] (0,0) ++(60*2+6:1cm) arc (60*2+6:360-(60*1+6):1cm);
\end{tikzpicture}
\end{center}
définissent un prographe $\mathcal{P}=(\Gamma_n, \theta_{mn})_{n \geq 1}$. 
Pour $m \geq n$, l'épimorphisme $\theta_{mn}$ est défini, pour $x \in \mathbb{Z}$, par :
\begin{center}
$\begin{array}{ccccc}
\theta_{mn} & : & \Gamma_m & \to & \Gamma_n \\
 & & P_x \ \mod \ p^m  & \mapsto & P_x \ \mod \ p^n \\
 & & Q_x \ \mod \ p^m  & \mapsto & Q_x \ \mod \ p^n \\
 & & [P_x \ \mod \ p^m, Q_x \ \mod \ p^m ]  & \mapsto &[ P_x \ \mod  \ p^n,  Q_x \ \mod \ p^n]  \\
  & & [P_{x+1} \ \mod \ p^m, Q_x \ \mod \ p^m ]  & \mapsto &[ P_{x+1} \ \mod \ p^n,  Q_x \ \mod \ p^n]  \\
\end{array}$
\end{center}
$\theta_{mn}$ est clairement surjectif.
\vskip 2mm
Pour tout $n \geq 1$, $D_{p^n}= \mathbb{Z}/p^n\mathbb{Z}\rtimes \mathbb{Z}/2\mathbb{Z}$ agit sur le graphe $\Gamma_n$, le segment $[P_{\overline 0},Q_{\overline 0}]$ est un domaine fondamental pour cette action. Soit $a \in \mathbb{Z}$, on note $\overline{a}$ sa classe modulo $p^n\mathbb{Z}$. Les éléments $D_{p^n}$ s'écrivent sous la forme $(\overline{a}, b)$, où $b \in  \mathbb{Z}/2\mathbb{Z}$. L'action de $D_{p^n}$ sur $\Gamma_n$ est donnée par :
 $$
(\overline{a},0)\cdot [P_{\overline{n}}-Q_{\overline{n}}] = \left\{
    \begin{array}{ll}
         [P_{\overline{n+k}}-Q_{\overline{n+k}}]& \mbox{si } \overline{a}=\overline{2k} \\
         \left[ P_{\overline{n+k+1}}-Q_{\overline{n+k}} \right ]& \mbox{si } \overline{a}=\overline{2k+1}
    \end{array}
\right.
$$
$$
(\overline{a},1)\cdot [P_{\overline{n}}-Q_{\overline{n}}] = \left\{
    \begin{array}{ll}
         [P_{\overline{n-k}}-Q_{\overline{n-k}}]& \mbox{si } \overline{a}=\overline{2k} \\
         \left[ P_{\overline{n-k}}-Q_{\overline{n-k-1}} \right ]& \mbox{si } \overline{a}=\overline{2k+1}
    \end{array}
\right.
$$
$$
(\overline{a},0)\cdot [P_{\overline{n+1}}-Q_{\overline{n}}] = \left\{
    \begin{array}{ll}
         [P_{\overline{n+k+1}}-Q_{\overline{n+k}}]& \mbox{si } \overline{a}=\overline{2k} \\
         \left[ P_{\overline{n+k+1}}-Q_{\overline{n+k+1}} \right ]& \mbox{si } \overline{a}=\overline{2k+1}
    \end{array}
\right.
$$
$$
(\overline{a},1)\cdot [P_{\overline{n+1}}-Q_{\overline{n}}] = \left\{
    \begin{array}{ll}
         [P_{\overline{n-k+1}}-Q_{\overline{n-k}}]& \mbox{si } \overline{a}=\overline{2k }\\
         \left[ P_{\overline{n-k+1}}-Q_{\overline{n-k+1}} \right ]& \mbox{si } \overline{a}=\overline{2k+1}
    \end{array}
\right.
$$ 
Le prographe $\mathcal{P}=(\Gamma_n, \theta_{mn})_n$ est en fait un proarbre. En effet on a $\limind_s \Gamma_n \sim \Gamma$ où s est la section graphiquement cohérente  arborescente définie ci-dessous .
\vskip 2mm
\noindent	
Pour $m>n$, on pose, pour $x \in \{0, \cdots, [p^n/2]\}$,
\begin{center} 
$\begin{array}{ccccc}
s_{nm} & : & \Gamma_n & \to & \Gamma_m \\
 & & P_x \ \mod \ p^n  & \mapsto & P_x \ \mod \ p^m \\
 & & Q_x \ \mod \ p^n  & \mapsto & Q_x \ \mod \ p^m \\
 & & P_{p^n-x} \  \mod \ p^n  & \mapsto & P_{p^n-x} \ \mod \ p^m \\
 & &Q_{p^n-x} \ \mod \ p^n  & \mapsto & Q_{p^n-x} \ \mod \ p^m \\
 & & [P_x \ \mod \ p^n- Q_x \ \mod \ p^n ]  & \mapsto &[ P_x  \ \mod \ p^m-  Q_x \ \mod \ p^m]  \\
  & & [P_{x} \ \mod \ p^n- Q_{x-1} \ \mod \ p^n ]  & \mapsto &[ P_{x} \ \mod  \ p^m-  Q_{x-1} \ \mod \ p^m]  \\
\end{array}$
\end{center}
et, pour $x \in \{1, \cdots, [p^n/2]\}$,
\begin{center} 
$\begin{array}{ccccc}
s_{nm} & : & \Gamma_n & \to & \Gamma_m \\
 & & [P_{p^n-x} \ \mod \ p^n- Q_{p^n-x} \ \mod \ p^n ]  & \mapsto &[ P_{p^m-x} \ \mod \ p^m-  Q_{p^m-x} \ \mod \ p^m]  \\
  & & [P_{p^n-x} \ \mod \ p^n- Q_{p^n-x-1} \ \mod \ p^n ]  & \mapsto &[ P_{p^m-x} \ \mod \ p^m-  Q_{p^m-x-1} \ \mod \ p^m]  \\
\end{array}$
\end{center}
Ces formules sont valables pour $p \geq 3$, pour $p=2$ on peut considérer les sections suivantes : On envoie $\Gamma_n$  sur la partie du dessus de $\Gamma_m$, pour $m \geq n$.\\
\vskip 2mm
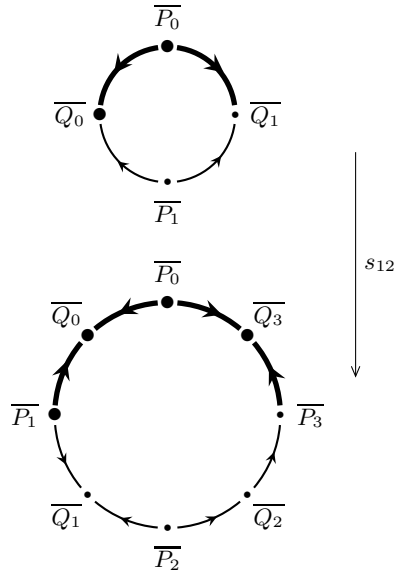
\noindent
\begin{center}
\begin{tikzpicture}[scale=1]
\foreach \x in {7}
  \path[draw,gros] (0,0) +(90+45*\x:15mm) node (sommet\x) {\hbox{\large $\bullet$}} +(90+45*\x:19mm) node {$\overline{Q_3}$} +(90+45*\x+20:16mm) node (arete\x) {}
    +(90+45*\x-45+5:15mm) arc (90+45*(\x-1)+5:90+45*\x-5:15mm);
\foreach \x in {0}
  \path[draw,gross] (0,0) +(90+45*\x:15mm) node (sommet\x) {\hbox{\large $\bullet$}} +(90+45*\x:19mm) node {$\overline {P_0}$} +(90+45*\x+20:16mm) node (arete\x) {}
    +(90+45*\x-45+5:15mm) arc (90+45*(\x-1)+5:90+45*\x-5:15mm);
\foreach \x in {1}
  \path[draw,gros] (0,0) +(90+45*\x:15mm) node (sommet\x) {\hbox{\large $\bullet$}} +(90+45*\x:19mm) node {$\overline {Q_0}$} +(90+45*\x+20:16mm) node (arete\x) {}
    +(90+45*\x-45+5:15mm) arc (90+45*(\x-1)+5:90+45*\x-5:15mm);
    \foreach \x in {2}
  \path[draw,gross] (0,0) +(90+45*\x:15mm) node (sommet\x) {\hbox{\large $\bullet$}} +(90+45*\x:19mm) node {$\overline {P_1}$} +(90+45*\x+20:16mm) node (arete\x) {}
    +(90+45*\x-45+5:15mm) arc (90+45*(\x-1)+5:90+45*\x-5:15mm);
\foreach \x in {3}
  \path[draw,moyen] (0,0) +(90+45*\x:15mm) node (sommet\x) {\hbox{\tiny $\bullet$}} +(90+45*\x:19mm) node {$\overline{Q_1}$} +(90+45*\x+20:16mm) node (arete\x) {}
    +(90+45*\x-45+5:15mm) arc (90+45*(\x-1)+5:90+45*\x-5:15mm);
\foreach \x in {4}
  \path[draw,moyenn] (0,0) +(90+45*\x:15mm) node (sommet\x) {\hbox{\tiny $\bullet$}} +(90+45*\x:19mm) node {$\overline{P_2}$} +(90+45*\x+20:16mm) node (arete\x) {}
    +(90+45*\x-45+5:15mm) arc (90+45*(\x-1)+5:90+45*\x-5:15mm);
\foreach \x in {5}
  \path[draw,moyen] (0,0) +(90+45*\x:15mm) node (sommet\x) {\hbox{\tiny $\bullet$}} +(90+45*\x:19mm) node {$\overline{Q_2}$} +(90+45*\x+20:16mm) node (arete\x) {}
    +(90+45*\x-45+5:15mm) arc (90+45*(\x-1)+5:90+45*\x-5:15mm);
    \foreach \x in {6}
  \path[draw,moyen] (0,0) +(90+45*\x:15mm) node (sommet\x) {\hbox{\tiny $\bullet$}} +(90+45*\x:19mm) node {$\overline{P_3}$} +(90+45*\x+20:16mm) node (arete\x) {}
    +(90+45*\x-45+5:15mm) arc (90+45*(\x-1)+5:90+45*\x-5:15mm);
\foreach \x in {0}
  \path[draw,gross] (0,4) +(90+90*\x:9mm) node (sommet\x) {\hbox{\large $\bullet$}} +(90+90*\x:13mm) node {$\overline{P_0}$} +(90+90*\x+20:10mm) node (arete\x) {}
    +(90+90*\x-90+8:9mm) arc (90+90*(\x-1)+8:90+90*\x-8:9mm);
\foreach \x in {1}
  \path[draw,gros] (0,4) +(90+90*\x:9mm) node (sommet\x) {\hbox{\large $\bullet$}} +(90+90*\x:13mm) node {$\overline{Q_0}$} +(90+90*\x+20:10mm) node (arete\x) {}
    +(90+90*\x-90+8:9mm) arc (90+90*(\x-1)+8:90+90*\x-8:9mm);
\foreach \x in {2}
  \path[draw,moyenn] (0,4) +(-270+90*\x:9mm) node (sommet\x) {\hbox{\tiny $\bullet$}} +(-270+90*\x:13mm) node {$\overline{P_1}$} +(-270+90*\x+20:10mm) node (arete\x) {}
    +(-270+90*\x-90+8:9mm) arc (-270+90*(\x-1)+8:-270+90*\x-8:9mm);
\foreach \x in {3}
  \path[draw,moyen] (0,4) +(-270+90*\x:9mm) node (sommet\x) {\hbox{\tiny $\bullet$}} +(-270+90*\x:13mm) node {$\overline{Q_1}$} +(-270+90*\x+20:10mm) node (arete\x) {}
    +(-270+90*\x-90+8:9mm) arc (-270+90*(\x-1)+8:-270+90*\x-8:9mm);
\draw [->,>=angle 45] (2.5,3.5) -- node [right] {$s_{12}$} (2.5,0.5);
\end{tikzpicture}
\end{center}
\vskip 2mm
\noindent
Considérons un indice $i \in \mathbb{N}$ et une arête $a \in   A(\Gamma_i)$. Si 
$$
a \neq \left\{
    \begin{array}{ll}
         [P_{[P^{i}/2]-1}, Q_{[P^{i}/2]]}& \mbox{si } p=2 \\
         \left[ Q_{ \left[ P^{i}/2\right]}, P_{\left[P^{i}/2 \right]+1} \right ]& \mbox{si } p \geq 3
    \end{array}
\right.
$$
alors l'indice $i_0=i$ convient pour $a$. Si maintenant 
$$
a = \left\{
    \begin{array}{ll}
         [P_{[P^{i}/2]-1}, Q_{[P^{i}/2]]}& \mbox{si } p=2 \\
         \left[ Q_{ \left[ P^{i}/2\right]}, P_{\left[P^{i}/2 \right]+1} \right ]& \mbox{si } p \geq 3
    \end{array}
\right.
$$
alors l'indice $i_0=i+1$ convient pour $a$.
La section $s$ est donc graphiquement cohérente. 
\vskip 2mm
\noindent

Pour $m > n$, on considère la section $r_{nm}$ de $\varphi_{mn}$, pour $x \in $ $\{0, \cdots, [p^n/2]\}$, définie par :
$$r_{nm}(x  \ \mod \ (p^n),a)= (x \ \mod \ (p^m),a)$$
$$r_{nm}(p^n-x \ \mod \ (p^n) ,a)=(p^m-x \ \mod \ (p^m),a).$$ 
Le système $(r_{nm})_n$ est alors une section dans le sens où $\limind_r D_{p^n}$ (isomorphe à $D_\infty=\mathbb{Z} \rtimes \mathbb{Z}/2 \mathbb{Z}$) est un sous groupe de $\widehat{D_{p^\infty}}$.
\vskip 2mm
\noindent
Le groupe $\widehat{D_{p^{\infty}}}$ agit à domaine fondamental segment sur un proarbre. On en déduit, d'après le théorème $3.7$ qu'il est presqu'amalgamé.
\vskip 2mm

Illustrons maintenant sur une situation galoisienne ce que nous venons de faire.
\vskip 2mm
\noindent
Considérons un système cohérent de racine primitive $p^n$-ième de l'unité $(\xi_{p^n})_n$, c'est à dire pour tout $\xi_{p^n}$ une racine primitive $p^n$-ième de l'unité telle que si $p^m=kp^n$ alors $=\xi_{p^n}=\xi_{p^m}^{k}$. Par exemple on considère $\xi_{p^n}=\exp(\frac{2i\pi}{p^n})$. 
\vskip2mm
\noindent
Pour tout entier naturel $n \geq 1$, on considère l'extension galoisienne $\mathbb{C}((t^{\frac{1}{p^{n}}}))/\mathbb{R}((t))$ du corps de séries de Laurent à coefficients dans $\mathbb{R}$. Son groupe de Galois : $G_n= \Gal( \mathbb{C}((t^{\frac{1}{p^{n}}}))/\mathbb{R}((t)))$ est isomorphe à $ D_{p^n}$. Plus précisément si l'on note $c$ la conjugaison complexe et soit l'automorphisme  $\sigma : \mathbb{C}((t^{\frac{1}{p^{n}}}) \to \mathbb{C}((t^{\frac{1}{p^{n}}})$ définie par : $\sigma(t^{\frac{1}{p^{n}}})= \xi_{p^n}t^{\frac{1}{p^{n}}}$, le groupe $G_n:= \Gal( \mathbb{C}((t^{\frac{1}{p^{n}}}))/\mathbb{R}((t)))$ est engendré par $c$ et $\sigma$.
\vskip 2mm

Soit le système projectif $(G_n, \res_{mn})_{n \in \mathbb{N^*}}$ où  $\res_{mn}$ désigne l'application de restriction. La $p$-partie du corps des série de Puiseux sur $\mathbb{C}$ correspond au corps $\Puis_p(\mathbb{C})=\bigcup_{n \in \N} \mathbb{C}((t^{\frac{1}{p^{n}}}))$. Il s'agit d'une extension galoisienne sur $\mathbb{R}((t))$ de groupe de Galois $\limproj_n G_n  \simeq \widehat{D_{p_{\infty}}}$. En particulier ce groupe est presqu'amalgamé. 
\vskip 2mm
\noindent
On peut voir cette dernière propriété directement grâce à l'arithmétique de ce corps : pour tout $n \geq 1$, on considère le graphe $\Gamma_n$ donné par :

$$\mathcal{S}(\Gamma_n)= \{ \xi_{2p^n}^{2k}t^{\frac{1}{p^{n}}}, \xi_{2p^n}^{2k+}t^{\frac{1}{p^{n}}} k=0, \cdots, p^{n}-1\}$$
$$\mathcal{A}(\Gamma_n)= \{ [\xi_{2p^n}^{2k}t^{\frac{1}{p^{n}}};\xi_{2p^n}^{2k+1}t^{\frac{1}{p^{n}}} ],[\xi_{2p^n}^{2k+2}t^{\frac{1}{p^{n}}};\xi_{2p^n}^{2k+1}t^{\frac{1}{p^{n}}} ], k=0, \cdots, p^{n}-1\}.$$
\vskip 2mm
\noindent
Pour $n \geq 1$ donné, on peut définir un épimorphisme de graphe donné par :
\begin{center}
$\begin{array}{ccccc}
\theta_{n} & : & \Gamma_n & \to & \Gamma_{n-1} \\
 & & \xi_{2p^n}^{2k}t^{\frac{1}{p^{n}}}  & \mapsto& (\xi_{2p^n}^{2k}t^{\frac{1}{p^{n}}})^p \\
 & & \xi_{2p^n}^{2k+1}t^{\frac{1}{p^{n}}}  &\mapsto & (\xi_{2p^n}^{2k+1}t^{\frac{1}{p^{n}}})^p \\
 & & [\xi_{2p^n}^{2k}t^{\frac{1}{p^{n}}};\xi_{2p^n}^{2k+1}t^{\frac{1}{p^{n}}} ]  & \mapsto &[(\xi_{2p^n}^{2k}t^{\frac{1}{p^{n}}})^p; (\xi_{2p^n}^{2k+1}t^{\frac{1}{p^{n}}})^p ] \\
   & & [\xi_{2p^n}^{2k+2}t^{\frac{1}{p^{n}}};\xi_{2p^n}^{2k+1}t^{\frac{1}{p^{n}}} ]  & \mapsto &[(\xi_{2p^n}^{2k+2}t^{\frac{1}{p^{n}}})^p; (\xi_{2p^n}^{2k+1}t^{\frac{1}{p^{n}}})^p ] \\
\end{array}$
\end{center}
Il est clair que $(\Gamma_n, \theta_n)_n$ forme un prographe et pour tout $n \geq 1$, le groupe $G_n$ agit naturellement sur $\Gamma_n$. L'action est celle donnée par l'image dans le corps $\mathbb{C}((t^{\frac{1}{p^{n}}}))$ par les automorphismes $c$ et $\sigma$ de $G_n=\Gal( \mathbb{C}((t^{\frac{1}{p^{n}}}))/\mathbb{R}((t)))$. 
\vskip 2mm
\noindent
On voit que cette action est exactement celle que nous avions décrite dans la première partie de cet exemple. On en déduit  que $\Gal(\Puis_p(\mathbb{C})/\mathbb{R}((t)))$ agit à domaine fondamental segment sur un proarbre et donc que ce groupe est presqu'amalgamé.
\vskip2mm
Cet exemple incite à considérer d'autres situations arithmétiques explicites pour lesquelles on ne sait pas déterminer le groupe de Galois mais pour lesquelles on pourrait obtenir la propriété d'être presqu'amalgamé.
\nocite{*}
\bibliographystyle{alpha} 
\bibliography{biblio} 
\vskip 2mm
\noindent
Laboratoire de Mathematiques Nicolas Oresme Normandie Univ, UNICAEN, CNRS 14032 Caen FRANCE 
\vskip 2mm
\noindent
E-mail address: Ndeye-coumba.sarr@unicaen.fr

\end{document}